\documentclass[12pt]{amsart}
\usepackage{amssymb,amsmath}
\usepackage{amsthm}
\usepackage{geometry}    
\usepackage{mathrsfs}

\usepackage{vmargin}

\setmargrb{1in}{1in}{1in}{1in}

\newtheorem{theorem}{Theorem}[section]
\newtheorem*{theorem*}{Theorem}
\newtheorem{lemma}[theorem]{Lemma}
\newtheorem{proposition}[theorem]{Proposition}
\newtheorem{corollary}[theorem]{Corollary}
\theoremstyle{definition}
\newtheorem{definition}[theorem]{Definition}
\newtheorem{remark}[theorem]{Remark}
\newtheorem{example}[theorem]{Example}
\newtheorem{conjecture}[theorem]{Conjecture}
\newtheorem{problem}[theorem]{Problem}

\begin{document}

\title{On a problem posed by Mahler}

\author{Diego Marques} 
\address{Department of Mathematics, University of Brasilia, Brasilia, DF, Brazil}

\author{Johannes Schleischitz} 
\address{Institute of Mathematics, Univ. Nat. Res. Life Sci. Vienna, 1180, Austria}


\begin{abstract}
E. Maillet proved that the set of Liouville numbers is preserved
under rational functions with rational coefficients. Based on this result,
a problem posed by Kurt Mahler is to investigate whether there exist
entire transcendental functions with this property or not. 
For large parametrized classes of Liouville numbers, we construct such functions and moreover we show that it can be constructed such that all their derivatives share this property.
We use a completely different approach than in a recent paper, where functions with
a different invariant subclass of Liouville numbers were constructed (though with no information on derivatives). More generally, we study the image of Liouville numbers under analytic functions, with particular attention to $f(z)=z^{q}$ where $q$ is a rational number.  
\end{abstract}

\maketitle

\noindent
{{\em Keywords}: Liouville numbers, transcendental function, exceptional set, continued fractions \\
\textit{Math Subject Classification AMS 2010:} Primary 11J04, Secondary 11J82, 11K60}

\vspace{4mm}

\section{Introduction} \label{ggg}

\subsection{Definitions} \label{definitions}
As usual, for a real number $\alpha$ we will write $\lfloor \alpha\rfloor$ for the largest integer
not greater than $\alpha$, $\lceil \alpha \rceil$ for the smallest integer not smaller than $\alpha$,
and $\{\alpha\}=\alpha-\lfloor \alpha\rfloor$. Moreover $\Vert \alpha\Vert$ will denote the 
distance from $\alpha$ to the closest integer, and we will write $A\asymp B$ if both $A\ll B$
and $B\ll A$ are satisfied. For a function $f:X\mapsto Y$ and a set $A\subseteq X$
we will write $f(A):=\{ f(x):x\in{A}\}$.

A transcendental function is defined as an analytic function $f(z)$ which is algebraically independent 
of its variable $z$ over some field.  
We will usually assume this field to be $\mathbb{C}$, and when at times we deal with $\overline{\mathbb{Q}}$
or $\mathbb{Q}$ instead this will be explicitly mentioned.
The complementary set of analytic functions $f$ that satisfy some polynomial identity 
$P(z,f(z))=0$ with $P\in{\mathbb{C}[X,Y]}$ (resp. $P\in{\overline{\mathbb{Q}}[X,Y]}$ or
$P\in{\mathbb{Q}[X,Y]}$) are called algebraic functions.
It is a widely known fact that the set of algebraic entire functions (over $\mathbb{C}$)
coincides with the set of complex polynomials $\mathbb{C}[X]$. The non-trivial inclusion
can be inferred from Great~Picard~Theorem, see Theorem~4.2 and Corollary~4.4 in~\cite{conway}. 

At the end of XIXth century, after the proof by Hermite and Lindemann of the transcendence of $e^{\alpha}$ for all nonzero algebraic $\alpha$, a question arose:

\textit{Does a transcendental analytic function usually take transcendental values at algebraic points?}

In the example of the exponential function $e^z$, the word ``usually"\ stands for avoiding the exception $z=0$. The set of the exceptions of this ``rule"\ was named by Weierstrass as {\em exceptional set} of a function $f$, which is defined as
\begin{equation} \label{eq:ef}
S_{f}:= \{\alpha\in{\overline{\mathbb{Q}}}: f(\alpha)\in{\overline{\mathbb{Q}}}\}.
\end{equation}

The study of exceptional sets started in 1886 with a letter of Weierstrass to
Strauss. Clearly, for algebraic functions over the field $\overline{\mathbb{Q}}$, one has $S_{f}= \overline{\mathbb{Q}}$. In 2009, Huang, Marques and Mereb \cite{huang} proved, in particular, that all subset of $\overline{\mathbb{Q}}$ is the exceptional set of uncountable many transcendental entire functions (including their derivatives), see \cite{hyper} for a more general result.

\subsection{Liouville numbers and Mahler's classification} \label{beginn} 
The {\em irrationality exponent}
of a real number $\alpha$, denoted by $\mu(\alpha)$, is defined as 
the (possibly infinite) supremum of all $\eta\geq 0$ such that
\begin{equation} \label{eq:1}
\left\vert \alpha-\frac{y}{x}\right\vert \leq x^{-\eta}
\end{equation}
has infinitely many rational solutions $y/x$. 
We point out that \eqref{eq:1} can be written equivalently using linear
forms as $\vert \alpha x-y\vert\leq x^{-\eta+1}$. Mostly in this paper, 
the linear form representation will be more convenient.
By Dirichlet's~Theorem, Corollary~2 in \cite{wald},
$\mu(\alpha)\geq 2$ for all $\alpha\in{\mathbb{R}}\backslash \mathbb{Q}$ and the equality holds for non rational real algebraic numbers $\alpha$ (by Roth's theorem), whereas $\mu(p/q)=0$.

Real numbers with irrationality exponent equal to infinity
are called {\em Liouville numbers}. We will write $\zeta$ for Liouville numbers
in contrast to $\alpha$ for arbitrary real numbers and
denote the set of Liouville numbers by $\mathscr{L}$. The elements of $\mathscr{L}$
are known to be transcendental
by Liouville's~Theorem, which also led to the first construction of
a transcendental number, namely the Liouville constant
\begin{equation} \label{eq:a}
L=\sum_{n\geq 1} 10^{-n!}=0.1100010000000000000000010\ldots.
\end{equation}
Altering the exponents
in $L$ slightly and adding fixed rational numbers it is not hard to construct uncountably
many elements of $\mathscr{L}$ within any set $A\subseteq \mathbb{R}$ with non-empty interior, 
see also Theorem~\ref{maillet} in Section~\ref{1.1}. Furthermore, 
the set $\mathscr{L}$ is known to be a dense $G_{\delta}$ set, since it can 
be written $\mathscr{L}=\cap_{n\geq 1} U_{n}$ where
\[
U_{n}:=\bigcup_{q\geq 2} \bigcup_{p\in{\mathbb{Z}}} 
\left(\frac{p}{q}-\frac{1}{q^{n}},\frac{p}{q}+\frac{1}{q^{n}}\right)\setminus \left\{\frac{p}{q}\right\}
\]
are open dense sets. Thus $\mathscr{L}$ is a residual set, i.e. the complement of a 
first category set.
However, $\mathscr{L}$ is very small in sense of measure theory, as its
Hausdorff dimension is $0$, see~\cite{jarnik}.

Some results of the paper are related to Mahler's $U$-numbers, so we want to give a short introduction
of Mahler's classification of real transcendental numbers into $S$,$T$ and $U$-numbers
regarding their properties concerning approximation
by algebraic numbers. 
In fact we will introduce Koksma's classes $S^{\ast},T^{\ast}$ and $U^{\ast}$, however the corresponding classes
are known to be pairwise identical~\cite[cf. Theorem~3.6]{ybu}. 
For real transcendental $\zeta$ and $n\geq 1$
an integer, define $w_{n}^{\ast}(\zeta)$ as the (possibly infinite) 
supremum of $\nu>0$ such that
\[
0<\vert \alpha-\zeta\vert \leq H(\alpha)^{-\nu-1}
\]
has infinitely many solutions in algebraic numbers $\alpha$ of degree at most $n$ for arbitrarily large $X$, 
where $H(\alpha)$ is the largest absolute value of the coefficients of
the irreducible (over $\mathbb{Z}$) minimal polynomial $P\in{\mathbb{Z}[X]}$ of $\alpha$. 
Obviously $w_{1}^{\ast}(\zeta)\leq w_{2}^{\ast}(\zeta)\leq \cdots$.
The set of $S$-numbers is defined as the set of real transcendental numbers that satisfy
\[
\limsup_{n\to\infty} \frac{w_{n}^{\ast}(\zeta)}{n}<\infty.
\]
$T$-numbers are defined by the properties
\[
\limsup_{n\to\infty} \frac{w_{n}^{\ast}(\zeta)}{n}=\infty,
\qquad w_{n}^{\ast}(\zeta)<\infty \quad \text{for} \quad n=1,2,\ldots.
\]
Finally $U$-numbers are defined as numbers that 
satisfy $w_{n}^{\ast}(\zeta)=\infty$ for some finite index $n$. 
If $m$ is the smallest such index then $\zeta$ is a $U_{m}$-number. The definitions imply that the set 
$\mathscr{L}$ coincides with the set of $U_{1}$-numbers, and the set of $U$-numbers
is the disjoint union of the sets of $U_{m}$-numbers over $m\geq 1$. 
We quote some important facts.
Two algebraically dependent numbers belong to the same class~\cite{burgertubbs},~\cite{leveque2}. 
Almost all $\zeta$, in the sense of Lebesgue measure, are $S$-numbers (this follows immediately from a result
of Sprind\^zuk~\cite{sprindi}), but the set of $T$-numbers
and $U_{m}$-numbers for $m\geq 1$ are non-empty. W. Schmidt was the first to construct $T$-numbers~\cite{wschm},
and the first construction of $U_{m}$-numbers of arbitrary prescribed degree $m$ was due to LeVeque~\cite{leveque}.
See also~\cite[Chapter~3]{schneider} or~\cite[Chapter~3]{ybu}. 

\subsection{The image of $\mathscr{L}$ under analytic functions}  \label{1.1}
In his pioneering book, E. Maillet \cite[Chapitre III]{maillet} discusses some arithmetic properties of Liouville numbers. In particular, he proved the following result concerning the image
of $\mathscr{L}$ under analytic functions.
\begin{theorem}[Maillet] \label{maillet}
If $f$ is non-constant rational function with rational coefficients, then $f(\mathscr{L})\subseteq \mathscr{L}$.
\end{theorem}
We observe that a kind of converse of this result is not valid in general, e.g., taking $f(x)=x^2$ 
and any number of the form $\ell=\sum a_{j}10^{-j!}$ with $a_{j}\in{\{2,4\}}$, the number $\zeta=\sqrt{(3+\ell)/4}$
is not a Liouville number \cite[Theorem 7.4]{ybu}, but $f(\zeta)$ is. Also the rational coefficients cannot be taken algebraic (with at least one of them non-rational). For instance, for $L$ in \eqref{eq:a} and $m\geq 2$
the number $L\sqrt[m]{3/2}$ is not a Liouville number, 
see~\cite[Th\' eor\` eme I$_3$]{maillet}. In fact, $L\sqrt[m]{3/2}$ is a $U_{m}$-number~\cite{CM}.

A problem posed by Mahler~\cite{mahler} is to study which analytic functions share this property.
In particular he asked whether there exist non-constant entire transcendental functions for which this is true.

In 1886, Weierstrass already made a construction of entire transcendental functions
with the property $f(\mathbb{Q})\subseteq \mathbb{Q}$. 
St\"ackel~\cite{staeckel} proved that for any countable set $A\subseteq\mathbb{C}$ and any dense set
$B\subseteq \mathbb{C}$, there exists an entire transcendental function $f$ with the 
property $f(A)\subseteq B$. F. Gramain showed that this is true for subsets of $\mathbb{R}$ as well.
Several other generalizations are known, we refer the reader 
to~\cite{huang},~\cite{marques},~\cite{mmoreira},~\cite{alf} for references.
However, due to the uncountable cardinality of $\mathscr{L}$, the used 
classic methods dealing with recursive constructions, do not to provide an obvious construction of
entire transcendental functions with $f(\mathscr{L})\subseteq \mathscr{L}$.
More generally, Mahler's problem suggests to study the set $f(\mathscr{L}\cap I)\cap \mathscr{L}$
for functions $f$ analytic on some interval $I\subseteq \mathbb{R}$ with real Taylor coefficients.
A recent result due to Kumar, Thangadurai and Waldschmidt admits to show that the set is always 
rather large, we will carry this out in Section~\ref{sek4}.

\subsection{Continued fractions}  \label{cfr}

We introduce the notation we will
use throughout the paper for continued fractions and 
gather various related results. The proofs can be found in~\cite{perron}
if not stated otherwise.

Let $\alpha\in{\mathbb{R}\setminus \mathbb{Q}}$. Let $\alpha_{0}=\alpha$, $r_{0}=\lfloor \alpha\rfloor$
and define the sequences $(r_{j})_{j\geq 0}$, $(\alpha_{j})_{j\geq 0}$ via
the recursive formulas $r_{j+1}=\lfloor 1/\{\alpha_{j}\}\rfloor$ and $\alpha_{j+1}=\{1/\{\alpha_{j}\}\}$.
Then if we define 
\[
[r_{0};r_{1},r_{2},\ldots,r_{n}]:=r_{0}+1/(r_{1}+1/(r_{2}+\cdots+1/r_{n}))\cdots) 
\]
the identity $\alpha=\lim_{n\to\infty} [r_{0};r_{1},r_{2},\ldots,r_{n}]$ 
holds. This representation is unique and $[r_{0};r_{1},r_{2},\ldots]$ is called 
the continued fraction expansion of $\alpha$, and $r_{j}$
are called {\em partial quotients}. Denote 
\[
\frac{s_{n}}{t_{n}}=[r_{0};r_{1},\ldots,r_{n}], \qquad \qquad n\geq 0,
\]
the $n$-th {\em convergent} of $\alpha$ in lowest terms. 
If we put $t_{-2}=1, t_{-1}=0$, we have 
\begin{equation} \label{eq:contf}
t_{n}=r_{n}t_{n-1}+t_{n-2}, \qquad n\geq 0.
\end{equation}
The analogue recursive formula for the $s_{n}$ holds but we do not need it. 
Moreover,  for any $n\geq 0$ we have $\vert s_{n}t_{n+1}-s_{n+1}t_{n}\vert=1$, such that 
both $s_{n},s_{n+1}$ such as $t_{n},t_{n+1}$ are coprime.

\begin{theorem}[Legendre] \label{prop}                     
Let $\alpha\in{\mathbb{R}\setminus{\mathbb{Q}}}$. 
If $\vert \alpha q-p\vert <(1/2)q^{-1}$ holds for integers $p,q$,
then the fraction $p/q$ equals a convergent of the continued fraction expansion of $\alpha$.
\end{theorem}

\begin{theorem}[Lagrange]  \label{cfe}                       
Let $\alpha\in{\mathbb{R}\setminus{\mathbb{Q}}}$ and $s_{n}/t_{n}$ the $n$-th convergent of 
$\alpha=[r_{0};r_{1},r_{2},\cdots]$. Then 
\[
\frac{r_{n+2}}{t_{n+2}}<\vert \alpha t_{n}-s_{n}\vert < \frac{1}{t_{n+1}}=
\frac{1}{t_{n}r_{n+1}+t_{n-1}}<\frac{1}{t_{n}r_{n+1}}.
\]
\end{theorem}

In particular, it follows from \eqref{eq:contf} that $\lim_{n\to\infty}\log r_{n+1}/\log t_{n}=\infty$ 
is equivalent to $\lim_{n\to\infty}\log t_{n+1}/\log t_{n}=\infty$, and 
in this case $\alpha\in{\mathscr{L}}$ follows.
More precisely, 
\[
\limsup_{n\to\infty} \frac{\log t_{n+1}}{\log t_{n}}=\infty \qquad \Longleftrightarrow \qquad \alpha\in{\mathscr{L}}.
\]

\subsection{Outline} \label{outline}
This paper is organized in the way that the Sections~\ref{present}, \ref{vieledefs}, \ref{vergangen} 
deal with the main topic of $f(\mathscr{L})\subseteq \mathscr{L}$ for
entire transcendental functions, whereas the Sections~~\ref{sek4},~\ref{hochab} discuss
related topics indicated in Section~\ref{definitions}. The assertion of our main result concerning the first category, 
Theorem~\ref{haupt}, at first sight appears similar to a recent result~\cite{marques}
which we will state in Section~\ref{vieledefs}.
We will show in Section~\ref{vieledefs}, though, that the classes considered in the respective 
theorems are in fact significantly different, and want to point out that also the proofs differ vastly.
Moreover, we point out the advantage of Theorem~\ref{haupt} that it makes assertions on the derivatives
too. See Remark~\ref{qbild} for another difference. Concerning results on related topics the main result
we will proof in Section~\ref{hochab} is basically the following.

\begin{theorem*}
For any $q\in{\mathbb{Q}\setminus \{0\}}$ let $f_{q}(z)=z^{q}$. 
Then there exist uncountably many $\zeta$, some of which 
can be explicitly constructed, such that $f_{q}(\zeta)\in{\mathscr{L}}$
if and only if $q\in{\mathbb{Z}}$. 
\end{theorem*}

\section{An approach connected to $f(\mathbb{Q})$}  \label{present}

For a function $f$ analytic in some open interval $I\subseteq \mathbb{R}$,
we will establish sufficient conditions for $f(\mathscr{L}\cap I)\subseteq \mathscr{L}$,
connected with the image $f(\mathbb{Q})$. More precisely, if we assume
$f(\mathbb{Q})\subseteq \mathbb{Q}$ as in various constructions, see Section~\ref{1.1}, 
and additionally assume certain upper bounds for the complexity of the fractions in the image,
we will be able to deduce $f(\mathscr{L}\cap I)\subseteq \mathscr{L}$.
Keep in mind that $I=\mathbb{R}$ leads to entire
functions. The method can be applied to confirm Theorem~\ref{maillet}.

\begin{theorem} \label{anfang}
Suppose $f$ is non-constant
analytic in some open interval $I\subseteq \mathbb{R}$ and $f(\mathbb{Q}\cap I)\subseteq \mathbb{Q}$.
Moreover, assume that there exists a function $\psi:\mathbb{R}_{>0}\mapsto \mathbb{R}_{>0}$ 
with the properties
\begin{itemize}
\item $\psi(m)=o(m)$ as $m\to\infty$
\item for $\zeta\in{\mathscr{L}}\cap I$ and any $m\geq 1$ we can find coprime $p_{m},q_{m}\geq 2$ such that
\begin{equation} \label{eq:wertz}
\left\vert \zeta-\frac{p_{m}}{q_{m}}\right\vert \leq q_{m}^{-m},
\end{equation}
and writing $f(p_{m}/q_{m})=p_{m}^{\prime}/q_{m}^{\prime}$ 
in lowest terms, we have $q_{m}^{\prime}\leq q_{m}^{\psi(m)}$.
\end{itemize}
Then $f(\mathscr{L}\cap I)\subseteq \mathscr{L}$.
\end{theorem}

\begin{proof}
Let $\zeta\in{\mathscr{L}}$ arbitrary. Let $J\subseteq I$ be non-empty and compact. 
Then $U:=\max_{z\in{J}} \vert f^{\prime}(z)\vert$ is well-defined.
Since $\zeta\in{\mathscr{L}}$ we can write 
\[
\zeta=\frac{p_{m}}{q_{m}}+\epsilon_{m}, \qquad \vert \epsilon_{m}\vert\leq \frac{1}{U}q_{m}^{-m}
\]
for any integer $m\geq 1$ with coprime integers $p_{m},q_{m}$ where $q_{m}>0$. 
Say $f(p_{m}/q_{m})=p_{m}^{\prime}/q_{m}^{\prime}$, and by assumption 
$q_{m}^{\prime}\leq q_{m}^{\psi(m)}$. Now for $m$ sufficiently large that $p_{m}/q_{m}\in{J}$
the intermediate value theorem of differentiation gives
\begin{equation} \label{eq:widerlegen}
\left\vert f(\zeta)-\frac{p_{m}^{\prime}}{q_{m}^{\prime}}\right\vert
=\vert f(\zeta)-f(p_{m}/q_{m})\vert\leq U\vert\epsilon_{m}\vert\leq q_{m}^{-m}\leq q_{m}^{\prime -m/\psi(m)}.
\end{equation}
Since $\psi(m)=o(m)$, we conclude $\mu(f(\zeta))=\infty$
with $\mu$ the irrationality exponent unless
$f(\zeta)\in\mathbb{Q}$. To exclude
$f(\zeta)\in{\mathbb{Q}}$, assume the opposite and write $f(\zeta)=l_{1}/l_{2}$.
Since $f$ is not constant in $I$, by the Identity Theorem for analytic functions, 
see Theorem 3.7 and Corollary~3.10 in~\cite{conway}, there exists some neighborhood $W\ni{\zeta}$                                          
of $\zeta$ such that $f(z)\neq f(\zeta)$ for $z\in{W}$. 
Since $p_{m}/q_{m}$ converges to $\zeta$ as $m\to\infty$, we infer
$f(p_{m}/q_{m})\neq f(\zeta)$ for large $m$.  
Thus
\[
\left\vert f(\zeta)-f(p_{m}/q_{m})\right\vert=
\left\vert f(\zeta)-\frac{p_{m}^{\prime}}{q_{m}^{\prime}}\right\vert= 
\left\vert \frac{l_{1}}{l_{2}}-\frac{p_{m}^{\prime}}{q_{m}^{\prime}}\right\vert
 \geq \frac{1}{q_{m}^{\prime}l_{2}},
\]
which contradicts \eqref{eq:widerlegen} for large $m$ since $\psi(m)=o(m)$.
\end{proof}

We check that, as indicated above, rational functions with rational coefficients satisfy the conditions
of Theorem~\ref{anfang}. Let $f$ be such a function and $p,q$ integers. Then we can write
\[
f(p/q)= \frac{P(p,q)}{Q(p,q)}=\frac{p^{\prime}}{q^{\prime}}
\]
with fixed polynomials $P,Q\in{\mathbb{Z}[X,Y]}$ and $p^{\prime},q^{\prime}\in{\mathbb{Z}}$. 
Consider $\zeta\in{\mathscr{L}}$
fixed and let $p=p_{m}, q=q_{m}$ satisfy \eqref{eq:wertz} and put $p^{\prime}=p_{m}^{\prime},
q^{\prime}=q_{m}^{\prime}$. First observe that we may assume $q^{\prime}\neq 0$ since
$\zeta$ is transcendental and there are only finitely many algebraic poles of $f$, so there is no pole of $f$ 
(and hence $f$ is analytic) in a neighborhood of $\zeta$.
From \eqref{eq:wertz} we deduce $\vert p_{m}-\zeta q_{m}\vert<1$ 
and thus $p_{m}\asymp q_{m}$ with implied constants depending on $\zeta, P,Q$ but not on $m$. 
It follows that $q_{m}^{\prime}\ll q_{m}^{k}$ where $k$ 
is the degree of $Q$ and again the implied constant depends on $\zeta, P, Q$ only. 
Hence the constant function $\psi(z)=k+1$ (or $\psi(z)=k+\epsilon$ for any $\epsilon>0$)
satisfies the conditions of Theorem~\ref{anfang}.

Considering constant functions $\psi(z)$, we 
stem a corollary from Theorem~\ref{anfang} whose conditions do
not explicitly involve $\zeta$ but are solely conditions on
the image $f(\mathbb{Q})$. 

\begin{corollary} \label{bfang}
Suppose $f$ is non-constant analytic in some open interval 
$I\subseteq \mathbb{R}$ and $f(\mathbb{Q}\cap I)\subseteq \mathbb{Q}$.
Moreover, assume that there exists $\eta\in{\mathbb{R}}$ such that
\[
f(p/q)=p^{\prime}/q^{\prime}
\]
implies $q^{\prime}\leq q^{\eta}$ provided $(p,q)=1$, $(p^{\prime},q^{\prime})=1$ 
and $q\geq 2$. Then $f(\mathscr{L}\cap I)\subseteq \mathscr{L}$.
\end{corollary}

\begin{proof}
Since $\zeta\in{\mathscr{L}}$,  for any $m\geq 1$ there exist $p_{m},q_{m}$ with \eqref{eq:wertz}.
Apply for any such choice Theorem~\ref{anfang} with the constant function $\psi(m)=\eta$.
\end{proof}

Incorporating the additional condition of 
Theorem~\ref{anfang} or Corollary~\ref{bfang} for transcendental functions
seems difficult with the common methods, as used for instance in~\cite{huang} or
\cite{marques}. In this context, Theorem~1.2 in~\cite{marques} 
asserts that there exist entire transcendental functions with $q^{\prime}<q^{8q^{2}}$ 
in the notation of Corollary~\ref{bfang}. See also Theorem~2 in~\cite{mmoreira}
for a related result concerning the image of algebraic numbers of bounded height
under certain entire transcendental functions.

\section{Special classes of Liouville numbers}  \label{vieledefs}
We define a few interesting subclasses of $\mathscr{L}$. The first one, which is new and
will be considered in the main result Theorem~\ref{haupt},
is parametrized by real functions.

\begin{definition} \label{psi}
Let $\Phi$ be the set of all functions
$\varphi:\mathbb{R}_{\geq 2}\mapsto \mathbb{R}_{\geq 2}$ which are non-decreasing and
satisfy $\lim_{x\to\infty} \varphi(x)=\infty$. 
For $\varphi\in{\Phi}$ define $\mathscr{L}_{\varphi}$ the (possibly empty)
subclass of $\zeta\in{\mathscr{L}}$
for which for any given positive integer $N$, the estimate
\begin{equation} \label{eq:lhs}
-\frac{\log \Vert \zeta q\Vert}{\log q} \geq N
\end{equation}
has an integer solution $q=q(N)$ with $2\leq q\leq \varphi(N)$. Similarly,
let $\mathscr{L}_{\varphi}^{\ast}\supset \mathscr{L}_{\varphi}$ be the set of $\zeta\in{\mathscr{L}}$ for
which the condition holds for all $N\geq N_{0}(\zeta)$. 
\end{definition}

\begin{remark} \label{haha}
Observe that by Theorem~\ref{prop}, for $N\geq 2$ the smallest $q$ for which \eqref{eq:lhs} holds 
equals some denominator $t_{n}$ of a convergent of $\zeta$.
\end{remark}

\begin{remark}
Only evaluations of $\varphi\in{\Phi}$ at integers will be of importance, so
we could alternatively work with sequences. For $\varphi\in{\Phi}$ of low growth, the sets
$\mathscr{L}_{\varphi},\mathscr{L}_{\varphi}^{\ast}$ are indeed empty.
However, we will see soon that the sets are large for $\varphi$ of sufficiently fast growth.
\end{remark}

Define orderings on $\Phi$ by $\psi\leq \varphi$ (resp. $\psi\leq_{\ast} \varphi)$ 
if $\psi(x)\leq \varphi(x)$ for all $x\geq 2$ (resp. $x\geq x_{0}=x_{0}(\varphi,\psi))$. 
These relations are clearly reflexive and transitive. 
The relation $\leq$ is also antisymmetric and hence $(\Phi,\leq)$ is a partially ordered set.
Furthermore, the pointwise maximum 
of two functions lies above both functions in these partial orders, 
such that $(\Phi,\leq)$ and $(\Phi,\leq_{\ast})$ can be viewed as directed sets. 
Obviously $\psi\leq \varphi$ implies $\mathscr{L}_{\psi}\subseteq \mathscr{L}_{\varphi}$ 
and $\psi\leq_{\ast} \varphi$ implies $\mathscr{L}_{\psi}^{\ast}\subseteq \mathscr{L}_{\varphi}^{\ast}$, 
such that the set of all $\{\mathscr{L}_{\varphi}\}$ (resp. $\{\mathscr{L}_{\varphi}^{\ast}\}$),
partially ordered by inclusion, are directed sets as well.
For any $\zeta\in{\mathscr{L}}$, say $\mathscr{A}(\zeta)\subseteq \Phi$ is the set
of $\varphi\in{\Phi}$ such that $\zeta\in{\mathscr{L}_{\varphi}}$.
There is a unique $\varphi\in{\mathscr{A}(\zeta)}$ 
with the property that $\varphi\leq \psi$ for any $\psi\in{\mathscr{A}(\zeta)}$ 
(in particular $\mathscr{A}(\zeta)\neq \emptyset$). This function
is locally constant, right-continuous, has image in $\mathbb{Z}_{\geq 2}$ 
and increases in a discontinuous way at integer values $q$ where an estimate $\Vert \zeta q\Vert\leq q^{-N}$
for some integer $N>0$ is satisfied for ''the first time'' (for $q$ but no smaller integer). 
We call it the {\em minimum function} for $\zeta\in{\mathscr{L}}$.

\begin{example}
For $L$ as in \eqref{eq:a} for any integer $n\geq 1$ we have 
\[
\Vert 10^{n!}L\Vert\leq 10^{n!-(n+1)!}+2\cdot 10^{n!-(n+2)!}=10^{-n\cdot n!}+2\cdot 10^{n!-(n+2)!}
\]
and hence
\[
-\frac{\log \Vert 10^{n!}L\Vert}{\log 10^{n!}}=\frac{n\cdot n!\log 10 }{n!\log 10}(1+o(1/n))=n+o(1).
\]
The remainder term tends to $0$ fast, such that certainly 
$\varphi(x)=10^{(x+1)!}$ is a proper choice for which $L\in{\mathscr{L}_{\varphi}}$, 
where we extend the definition of the factorials 
to real numbers by $x!:=x(x-1)(x-2)\cdots(1+\{x\})$. 
\end{example}

\begin{example} \label{bsp}
For either $\varphi(x)=2^{(x!)!}$ or $\varphi(x)=2^{2^{x!}}$, it is easy to check 
all numbers of the form $L_{M}:=\sum_{j\geq 1} M^{-j!}$ for $M\geq 2$ an integer,
belong to $\mathscr{L}_{\varphi}^{\ast}$ simultaneously.
\end{example}

\begin{proposition} \label{uncountable}    
Let $\varphi\in{\Phi}$ for which $\mathscr{L}_{\varphi}\neq \emptyset$,
for example the minimum function of arbitrary $\zeta\in{\mathscr{L}}$.
Then the set $\mathscr{L}_{\varphi}$ is uncountable. Moreover,
for any non-empty open interval $J$ the set $\mathscr{L}_{\varphi}^{\ast}\cap J$ is uncountable.
\end{proposition}

\begin{proof}    
Say $\zeta=[r_{0};r_{1},r_{2},\ldots]$ belongs to $\mathscr{L}_{\varphi}$. By the properties
we have established, we may assume $\varphi$ is the minimum function of $\zeta$.
                                              
By Remark~\ref{haha}, any rise of the locally constant minimum function of $\zeta$
is induced by some convergent (in general not every convergent induces a rise). It is also obvious
that there are infinitely many rises since $\zeta\in{\mathscr{L}}$. 
Define the subsequence $j(n)$ of $\{0,1,2,\ldots\}$ such that
the $n$-th rise is induced by $s_{j(n)}/t_{j(n)}=[r_{0},r_{1},\ldots,r_{j(n)}]$,
i.e. $q=t_{j(n)}$ but no smaller integer satisfies \eqref{eq:lhs} for some integer $N$.
Then $r_{j(n)+1}$ is large.
For suitable subsets $T\subseteq\{j(1),j(2),\ldots\}$ the numbers
$\zeta_{T}$ defined as the number that arises from $\zeta$ by deleting precisely those 
partial quotients $r_{i}$ for which $i-1\in{T}$ will satisfy the claim.
We distinguish two cases. Case 1: There are arbitrarily large $i$ such that $j(i+1)>j(i)+1$ strictly.
Then we allow to delete those coefficients $j(i)+1$, i.e. $j(i)\in T$, for which this inequality holds.
Moreover we do not delete $j(i)+1$ for all large $i$, i.e. $T^{c}$ is infinite.
By virtue of \eqref{eq:contf} and Theorem~\ref{cfe} and since $T^{c}$ is infinite,
one checks that $\zeta_{T}\in{\mathscr{L}}$. On the other hand, the recurrence
\eqref{eq:contf} implies $\psi\leq \varphi$ for $\psi$ the minimum function of $\zeta_{T}$. 
Hence $\zeta_{T}\in{\mathscr{L}_{\psi}}\subseteq{\mathscr{L}_{\varphi}}$. Since there are uncountably
many choices for $T$ and the continued fraction expansion is uniquely determined, 
this yields uncountably many elements in $\mathscr{L}_{\varphi}$.
Case 2: For all sufficiently large $i\geq i_{0}$ we have $j(i+1)=j(i)+1$. In this case one may 
delete any subset of partial quotients of index greater than $i_{0}$. The properties
can be inferred similarly as in case 1 with \eqref{eq:contf}. 
The assertion on $\mathscr{L}_{\varphi}^{\ast}\cap J$ can be inferred from the above by altering
initial partial quotients, which only yields a rational transformation of $\zeta$.   
\end{proof}

Unfortunately, for any given $\varphi\in{\Phi}$
it is not hard to construct continued fraction expansions of
elements in $\mathscr{L}\setminus \mathscr{L}_{\varphi}$ either, 
such that $\mathscr{L}_{\varphi}\subsetneq \mathscr{L}$. 
It suffices to choose many successive small
partial quotients between rather large ones, such that the maximum of 
the left hand side in \eqref{eq:lhs} for bounded $q$
tends to infinity slower than $\varphi$. More generally, a diagonal method argument shows that 
there is no representation of $\mathscr{L}$ as a countable union of classes $\mathscr{L}_{\varphi}$.
However, obviously $\mathscr{L}$ can be written as the uncountable union 
$\cup_{\zeta\in{\mathscr{L}}} \mathscr{L}_{\varphi(\zeta)}$ where
$\varphi(\zeta)$ is the minimum function of $\zeta\in{\mathscr{L}}$.

We compare the classes $\mathscr{L}_{\varphi}$ with certain other subclasses 
of $\mathscr{L}$ that have been studied.
LeVeque~\cite{leveque} introduced strong Liouville numbers.
This concept was refined by Alniacik~\cite{alniacik}
who defined semi-strong Liouville numbers. The following definition comprises 
these concepts and some additional ones that fit our purposes.

\begin{definition} \label{defilio}
For $\zeta\in{\mathscr{L}}$ denote $s_{n}/t_{n}$ ($n\geq 0$) the sequence of its convergents.
Then $\zeta$ is called {\em semi-strong} if  
one can find a subsequence $(v_{i})_{i\geq 0}$ of $\{0,1,2\ldots\}$ with the properties
\begin{align}
&\left\vert t_{v_{i}}\zeta-s_{v_{i}}\right\vert= t_{v_{i}}^{-\omega(v_{i})}, 
\qquad \lim_{i\to\infty} \omega(v_{i})=\infty, \label{eq:ast}  \\
&\limsup_{i\to\infty} \frac{\log t_{v_{i+1}}}{\log t_{v_{i}+1}}<\infty. \label{eq:rein}
\end{align}
It is called {\em strong} if \eqref{eq:ast} is true for $v_{i}=i$ 
(note \eqref{eq:rein} is trivial then). 
Denote the sets of semi-strong (resp. strong) Liouville numbers by $\mathscr{L}^{ss}$ (resp. $\mathscr{L}^{s})$. 
Further for any non-decreasing function $\Lambda: \mathbb{R}_{\geq 1}\mapsto \mathbb{R}_{\geq 1}$ 
with $\lim_{x\to\infty} \Lambda(x)=\infty$, let $\mathscr{L}^{s,\Lambda}\subseteq \mathscr{L}^{s}$ 
(resp. $\mathscr{L}^{ss,\Lambda}\subseteq \mathscr{L}^{ss}$) be the sets 
for which $\omega(v_{i})\geq \Lambda(i)$ for some sequence $(v_{i})_{i\geq 1}$
as above. 
\end{definition}

Conversely to the sets $\mathscr{L}_{\varphi}$, the sets $\mathscr{L}^{s,\Lambda}$ and $\mathscr{L}^{ss,\Lambda}$
get smaller the faster $\Lambda$ tends to infinity. For any $\Lambda$ as in Definition~\ref{defilio}, 
choosing the partial quotients sufficiently large, it is easy to check  all defined sets 
are non-empty (in fact uncountable).

It is not hard to see $\mathscr{L}^{s}\subsetneq\mathscr{L}^{ss}\subsetneq \mathscr{L}$.
Unfortunately (in view of Section~\ref{constset}), for any given $\varphi\in{\Phi}$, there exist (semi-)strong 
Liouville numbers not contained in $\mathscr{L}_{\varphi}^{\ast}$, i.e. 
$\mathscr{L}^{s}\nsubseteq \mathscr{L}_{\varphi}^{\ast}$. To ensure inclusion we need
some (arbitrarily weak) additional minimum growth condition on the sequence $\omega(v_{i})$ in \eqref{eq:ast}.

\begin{proposition} \label{hilfspro}
Fix any function $\Lambda$ as in Definition~{\upshape\ref{defilio}}. Then
there exists $\varphi=\varphi(\Lambda)\in{\Phi}$ such that
$\mathscr{L}^{s,\Lambda}\subseteq \mathscr{L}_{\varphi}$.
Furthermore, there exists $\psi=\psi(\Lambda)\in{\Phi}$ for which
$\mathscr{L}^{s,\Lambda}\subseteq \mathscr{L}^{ss,\Lambda}\subseteq \mathscr{L}_{\psi}^{\ast}$.
\end{proposition} 

\begin{proof}   
First we construct $\varphi\in{\Phi}$ such that $\mathscr{L}^{s,\Lambda}\subseteq \mathscr{L}_{\varphi}$
and prove this rigorously, subsequently we sketch how to derive 
the other inclusion in a similar way. 

Consider an arbitrary but fixed integer $N\geq 1$. 
We will construct suitable $\varphi(N)$.
Let $\iota_{N}:=\lceil \Lambda^{-1}(N)\rceil$, i.e. 
the smallest index $i$ such that $\Lambda(i)\geq N$.
Consider integers $T_{1},\ldots,T_{\iota_{N}}$ given by the recurrence 
relation $T_{0}=1, T_{1}=N+1$ and $T_{j+1}= T_{j}^{N+1}$
for $1\leq j\leq \iota_{N}-1$ and put $D_{N}:=T_{\iota_{N}}$. We show that $\varphi(N):=D_{N}$
is a suitable choice. We use the notation of Section~\ref{cfr} for the continued
fraction expansion of $\zeta$. First assume all partial
denominators $t_{1},\ldots,t_{N}$ of the convergents of some $\zeta$ are bounded 
by $t_{j}\leq T_{j}$.
It follows from \eqref{eq:contf} that $t_{\iota_{N}}\leq T_{\iota_{N}}=D_{N}$, but on the
other hand the inequality $\vert t_{j}\zeta-s_{j}\vert<t_{j}^{-N}$ is satisfied for the index $j=\iota_{N}$ 
by definition of $\iota_{N}$. Thus if we put $q=T_{\iota_{N}}$ in Definition~\ref{psi} we see
$\varphi(N):=D_{N}$ is indeed a proper choice.
On the other hand, if for some $1\leq j\leq \iota_{N}-1$ we have $t_{j}>T_{j}$, 
then again by \eqref{eq:contf} and Theorem~\ref{cfe} we infer $\vert t_{j-1}\zeta-s_{j-1}\vert<t_{j-1}^{-N}$, 
and if $j$ is the smallest such index then moreover $t_{j-1}\leq D_{N}$.
Again this shows we may put $q=q(N)=t_{j-1}$ in Definition~\ref{psi} and $\varphi(N):=D_{N}$ is a proper choice.

For the inclusion $\mathscr{L}^{ss,\Lambda}\subseteq \mathscr{L}_{\psi}^{\ast}$ 
construct $\psi(N)=D_{N}$ as above, with the replacement $T_{j+1}:=T_{j}^{j(j+1)}$ in each inductive step. 
Observe that for any $\zeta\in{\mathscr{L}^{ss}}$, condition \eqref{eq:rein} guarantees 
we will have $t_{\iota_{N}}<T_{\iota_{N}}=:D_{N}$ for sufficiently large $N=N(\zeta)$. 
\end{proof}

Conversely, it can be shown that for any fixed $\varphi\in{\Phi}$ we 
have $\mathscr{L}_{\varphi}\nsubseteq \mathscr{L}^{ss}$.
We will not need this, though. Another subclass of $\mathscr{L}$ was recently defined  
in~\cite{marques}.

\begin{definition}
Recursively define $\exp^{[0]}(x):=x$ and $\exp^{[k+1]}(x)=\exp(\exp^{[k]})(x)$.
Then $\zeta\in{\mathscr{L}}$ is called {\em ultra-Liouville} if for any $k\geq 0$ there
exists a rational number $p/q$ such that
\begin{equation} \label{eq:ultra}
\left\vert \zeta-\frac{p}{q}\right\vert \leq \frac{1}{\exp^{[k]}(q)}.
\end{equation}
We denote the set of ultra-Liouville numbers by $\mathscr{L}_{ultra}$.
\end{definition}

Theorem~1.1 in~\cite{marques}, which relies on Theorem~1.2 in~\cite{marques}
mentioned at the end of Section~\ref{present}, asserts the following.

\begin{theorem}[Marques, Moreira]   \label{moreira}
There exist uncountably many entire transcendental functions $f$ such that 
$f(\mathscr{L}_{ultra})\subseteq \mathscr{L}_{ultra}$. In particular 
$f(\mathscr{L}_{ultra})\subseteq \mathscr{L}$.
\end{theorem}

It is important to get noticed that the previous result is strong in the sense that it ensures the existence of an uncountable subset of Liouville numbers which is invariant for uncountable many transcendental analytic functions. 

It is not hard to check that there exist functions $\varphi\in{\Phi}$ 
for which $\mathscr{L}_{\varphi}\nsubseteq \mathscr{L}_{ultra}$. 
It suffices to take $\varphi$ the minimum function of any $\zeta\in{\mathscr{L}}$
for which we cannot find a rational for which \eqref{eq:ultra} holds for $k=1$ (or any larger $k$), 
which clearly exists. Conversely, one checks 
$\mathscr{L}_{ultra}\nsubseteq \mathscr{L}_{\varphi}^{\ast}$ for any fixed
function $\varphi\in{\Phi}$, as the frequency of values $q$ inducing very good approximations $p/q$ in \eqref{eq:ultra}
can be arbitrarily low. For similar reasons, also combination of the concepts (semi-)strong and ultra
is not sufficient to provide $\varphi$ with inclusion, in other words 
$\mathscr{L}_{ultra}\cap \mathscr{L}^{s}\nsubseteq \mathscr{L}_{\varphi}$ for any $\varphi\in{\Phi}$.
Moreover, there is no inclusion within $\mathscr{L}^{s}$ (resp.
$\mathscr{L}^{ss}$) and $\mathscr{L}_{ultra}$. 
Finally, we also want to refer to~\cite{mmoreira} for
a result similar to Theorem~\ref{moreira} concerning the image of more general  
sets (in general no longer subsets of $\mathscr{L}$). 
There is again no immediate correlation to Theorem~\ref{haupt}. 

\section{Entire transcendental functions with large invariant set}  \label{vergangen}

\subsection{Preparatory results}
We put our focus on entire functions $f$ now. 
We gather some results that we will utilize in the proof of Theorem~\ref{haupt}.
The following Lemma~\ref{grobabschaetzung} on its own leads to another proof 
of Theorem~\ref{maillet} in the special case of polynomials. In the proof 
we will use the following elementary fact.
For a real number $\alpha$ and a positive integer $k$ the estimate
\begin{equation} \label{eq:lio2}
\vert q\alpha-p\vert\leq q^{-\nu}
\end{equation}
implies
\begin{equation} \label{eq:lio}
\vert q^{k}\alpha^{k}-p^{k}\vert=\vert q\alpha-p\vert\cdot\vert q^{k-1}\alpha^{k-1}+\cdots+p^{k-1}\vert
\leq D(k,\alpha)q^{-\nu+k-1}
\end{equation}
with a constant $D(k,\alpha)$ depending only on $k$ and $\alpha$. 
This argument was actually used in a slightly more general way in the proof of
Lemma~1 in~\cite{bug} and will be frequently applied in Section~\ref{hochab} as well. 

\begin{lemma}  \label{grobabschaetzung}
Let $\alpha\in{\mathbb{R}}$ and $P\in{\mathbb{Q}[X]}$ given as 
\[
P(z)=\frac{a_{0}}{b_{0}}+\frac{a_{1}}{b_{1}}z+\cdots+\frac{a_{m}}{b_{m}}z^{m}
\]
with $a_{j}/b_{j}$ in lowest terms. 
Put $A:=\max_{0\leq j\leq m} \vert a_{j}\vert, B:=$\rm{lcm}($\vert b_{0}\vert,\ldots,\vert b_{m}\vert)$.
Assume for a positive integer $q$ and {\upshape(}large{\upshape)} $\nu>0$ we have 
\begin{equation}  \label{eq:qbe}
\Vert q\alpha\Vert \leq q^{-\nu}.
\end{equation}
Then $Bq^{m}\in{\mathbb{Z}}$ and 
\[
\Vert Bq^{m}\cdot P(\alpha)\Vert \leq m^{2}(1+\vert\alpha\vert)^{m-1}\cdot ABq^{-\nu+m-1}.
\]
\end{lemma}

\begin{proof}
By definition $d_{k}:=\vert B/b_{k}\vert$ is an integer with $1\leq d_{k}\leq B$ for $0\leq k\leq m$. 
Recall that for any integer $M$ and $\alpha\in{\mathbb{R}}$ we have 
$\Vert M\alpha\Vert\leq \vert M\vert\cdot \Vert \alpha\Vert$.
For $0\leq k\leq m$ we estimate the monomial 
\begin{equation} \label{eq:langwirds}
\left\Vert Bq^{m}\frac{a_{k}}{b_{k}}\alpha^{k}\right\Vert=\left\Vert d_{k}a_{k}q^{m}\alpha^{k}\right\Vert 
\leq \vert a_{k}\vert d_{k}q^{m-k}\left\Vert q^{k}\alpha^{k}\right\Vert\leq ABq^{m-k}\left\Vert q^{k}\alpha^{k}\right\Vert.
\end{equation}
Moreover, for $k=0$ the left hand side of \eqref{eq:langwirds} is $0$,
which will improve the result slightly.
As $\nu$ is large and thus $p/q$ is very close to $\alpha$ for some $p\in{\mathbb{Z}}$,
we may apply \eqref{eq:lio} to estimate $\Vert q^{k}\alpha^{k}\Vert$ with the bound 
$D(k,\alpha)\leq k(1+\vert \alpha\vert)^{k-1}\leq m(1+\vert \alpha\vert)^{m-1}$ 
for any $1\leq k\leq m$. Since 
$\Vert \mu_{0}+\cdots+\mu_{m}\Vert\leq \Vert \mu_{1}\Vert+\cdots+\Vert \mu_{m}\Vert$ 
for all real $\mu_{0},\mu_{1},\ldots,\mu_{m}$ with $\mu_{0}\in\mathbb{Z}$,
we infer the lemma if we put $\mu_{k}$ the 
left hand sides of \eqref{eq:langwirds} for $0\leq k\leq m$.
\end{proof}

We will need an additional technical coprimeness result for special choices of coefficients $c_{j}$ 
in Lemma~\ref{grobabschaetzung} in the proof of Theorem~\ref{haupt}. 

\begin{proposition}  \label{relativprim}
Let $\alpha\in{\mathbb{R}}$ and $P\in{\mathbb{Q}[X]}$ as in Lemma~{\upshape\ref{grobabschaetzung}}
where $c_{j}=1/b_{j}$ and $b_{j}\vert b_{j+1}$ for $0\leq j\leq m-1$. Define $A,B$ as
in the lemma, such that $A=1, B=b_{m}$. 

There exists $\nu_{0}=\nu_{0}(P)$ which depends on $P$ but not on $q$, 
such that if $q\geq 2$ satisfies \eqref{eq:qbe} for $\nu\geq \nu_{0}$, 
and if for $p$ the closest integer to $q\alpha$
we have $(p,q)=1$, and if $R$ denotes the closest integer to $Bq^{m}\cdot P(\alpha)$, we have $(q,R)=1$.
\end{proposition}

\begin{proof}
There exists some large $\nu_{1}=\nu_{1}(P)$ independent from $q$ such that for $\nu\geq \nu_{1}$, 
all left hand sides in \eqref{eq:langwirds} in the proof of Lemma~\ref{grobabschaetzung}
are sufficiently small to add up to a number smaller than $1/2$. 
Then $R$ equals the sum of the $m+1$ closest integers to the monomials
$Bq^{m}a_{k}/b_{k}\alpha^{k}$, call them $Z_{k}$. 
In view of \eqref{eq:lio}, we have
\[
q^{m}\alpha^{k}=q^{m-k}(q\alpha)^{k}=q^{m-k}p^{k}+q^{m-k}\Vert q\alpha\Vert^{k}
\]
is very close to $q^{m-k}p^{k}$ uniformly in $0\leq k\leq m$, provided 
$\Vert q\alpha\Vert$ is sufficiently small. More precisely, it is not hard to check
that if $\nu$ in \eqref{eq:qbe} satisfies $\nu\geq \nu_{2}$ with large $\nu_{2}=\nu_{2}(P)$
independent from $q$, again writing $d_{k}=B/b_{k}\in{\mathbb{Z}}$ for $0\leq k\leq m$, we have
\[
Z_{k}=q^{m-k}p^{k}a_{k}d_{k}=q^{m-k}p^{k}d_{k}, \qquad 0\leq k\leq m.
\]
Note that $d_{m}=1$ since $b_{m}=B$ follows from the divisibility 
conditions on the $b_{j}$. Combining these results, if we let 
$\nu\geq \nu_{0}$ in \eqref{eq:qbe} with $\nu_{0}:=\max\{\nu_{1},\nu_{2}\}$, we infer
\[
R=Z_{0}+\cdots+Z_{m}=q^{m}d_{0}+q^{m-1}pd_{1}+q^{m-2}p^{2}d_{2}+\cdots+qp^{m-1}d_{m-1}+p^{m}.
\]
Clearly, any prime divisor of $q$ divides any other expression in the sum but 
certainly not $p^{m}$ since $(p,q)=1$ by assumption. The assertion follows. 
\end{proof}

\subsection{The main result} \label{constset}
Now we state the main theorem, which provides non-constant entire transcendental functions $f$ 
that map large prescribed subclasses of $\mathscr{L}$ to $\mathscr{L}$. It will turn out that
all derivatives have the same property.
The idea is to look at entire functions whose Taylor coefficients decrease fast
by absolute value, in order to apply Lemma~\ref{grobabschaetzung} with gain.
To exclude the case that an element of the image is rational is slightly technical.
We agree that $f(\emptyset)=\emptyset$ in the trivial case $\mathscr{L}_{\varphi}^{\ast}=\emptyset$.

\begin{theorem} \label{haupt}
Let $\varphi\in{\Phi}$ be arbitrary but fixed.
Then, there exist uncountably many entire transcendental functions $f(z)=c_{0}+c_{1}z+\cdots$ with 
$c_{j}\in{\mathbb{Q}\setminus{\{0\}}}$ and the property that for any $s\in{\{0,1,2,\ldots\}}$ we have
\begin{itemize}
\item  $f^{(s)}(0)\in{\mathbb{Q}}$
\item $f^{(s)}(\mathbb{Q}\setminus\{0\})\subseteq \mathscr{L}$
\item $f^{(s)}(\mathscr{L}_{\varphi}^{\ast})\subseteq \mathscr{L}$.
\end{itemize}
Suitable functions $f$ can be explicitly constructed.
\end{theorem}

\begin{proof}
First we prove the assertion for $s=0$, and subsequently describe 
how the proof extends to $s>0$.

Let $(T_{m})_{m\geq 1}$ be any sequence of positive real numbers that 
tends to infinity, for instance $T_{m}=m$. We recursively construct the
rational Taylor coefficients $c_{j}$ of suitable functions $f$.
Note that the first assertion of the theorem will follow immediately from 
$c_{j}\in{\mathbb{Q}}$. Let $c_{0}=1$. 
Assume the Taylor polynomial $P_{m}(z)=c_{0}+c_{1}z+\cdots+c_{m}z^{m}$ 
of $f$ of degree $m\geq 0$ is already constructed
and has rational coefficients $c_{j}=1/b_{j}$ and $b_{j}\vert b_{j+1}$ for $0\leq j\leq m-1$,
as in Proposition~\ref{relativprim}. We construct $c_{m+1}$.
Let $P:=P_{m}$ in Lemma~\ref{grobabschaetzung} and
similarly define $A:=A_{m}, B:=B_{m}$ with $A_{m}, B_{m}$ arising
from the present $a_{j},b_{j}$ as in the lemma. In fact, the conditions show $A_{m}=1, B_{m}=b_{m}$.
Let the positive integer $k_{m}$ be large enough such that 
\begin{equation} \label{eq:sc}
q^{k_{m}}>m^{2}(T_{m}+1)^{m-1}A_{m}B_{m}q^{m-1}\cdot 2(B_{m}q^{m})^{m}=:q^{m^{2}+m-1}D_{m}
\end{equation}
for any integer $q\geq 2$, which is possible since $D_{m}$ and the exponent 
$m^{2}+m-1$ are constants. Since we can make $k_{m}$ larger if necessary, we may assume 
$k_{m}\geq \nu_{0}(P_{m})$, where $\nu_{0}(P_{m})$ is as in Proposition~\ref{relativprim}
for $P=P_{m}$. By definition of the set $\mathscr{L}_{\varphi}$,
for any $\zeta\in{\mathscr{L}_{\varphi}}$ the inequality
\begin{equation} \label{eq:mitte}
\Vert q\zeta\Vert \leq q^{-k_{m}}
\end{equation}
has a solution $q=:\widetilde{q}_{m}$, that may depend on $\zeta$ but
with $2\leq \widetilde{q}_{m}\leq \varphi(k_{m})$ uniformly.
First consider only the class $\mathscr{L}_{\varphi}$ instead of $\mathscr{L}_{\varphi}^{\ast}$.
Restricting to $\zeta\in{\mathscr{L}_{\varphi}\cap[-T_{m},T_{m}]}$, application
of Lemma~\ref{grobabschaetzung} with $\nu:=k_{m}$ in view of \eqref{eq:sc} 
yields
\begin{equation} \label{eq:eq1}
\Vert (B_{m}\widetilde{q}_{m}^{m})\cdot P_{m}(\zeta)\Vert 
\leq  m^{2}(1+\vert\zeta\vert)^{m-1}\cdot A_{m}B_{m}\widetilde{q}_{m}^{-k_{m}+m-1}
\leq \frac{1}{2}\vert B_{m}\widetilde{q}_{m}^{m}\vert^{-m}.
\end{equation}
Put $\widetilde{Q}_{m}:=B_{m}\widetilde{q}_{m}^{m}$, then \eqref{eq:eq1} turns into
\begin{equation} \label{eq:drauf}
\Vert\widetilde{Q}_{m}P_{m}(\zeta)\Vert \leq \frac{1}{2}\widetilde{Q}_{m}^{-m}.
\end{equation}
Moreover, if we write $\tau_{m}:=B_{m}\varphi(k_{m})^{m}$, then we have 
\begin{equation} \label{eq:4}
\vert \widetilde{Q}_{m}\vert\leq \tau_{m}.
\end{equation}

Now we determine $c_{m+1}\in{\mathbb{Q}\setminus{\{0\}}}$ of very small modulus. 
Assume the coefficients $c_{m+2},c_{m+3},\ldots$ do not vanish but
are of very small and fast decreasing modulus too. More precisely, for now we 
assume all the coefficients $c_{m+1},c_{m+2},\ldots$ satisfy
\begin{equation} \label{eq:hbed}
\vert c_{m+h}\vert<\min\{(1/4)(1+T_{m})^{-m-2h}\tau_{m}^{-m-1},1/(m+h)!\}, \qquad h\geq 1,
\end{equation}
where the purpose of $1/(m+h)!$ is solely to guarantee convergence. Pick any suitable 
$c_{m+1}=1/b_{m+1}\in{\mathbb{Q}\setminus{\{0\}}}$ for $b_{m+1}$ a sufficiently large 
integral multiple of $b_{m}$ such that \eqref{eq:hbed} is satisfied for $h=1$. Then 
\[
\vert f(z)-P_{m}(z)\vert =\left\vert \sum_{h=1}^{\infty} c_{m+h}z^{m+h}\right\vert
\leq \sum_{h=1}^{\infty}\vert c_{m+h}\vert T_{m}^{m+h}
< \frac{1}{2}\tau_{m}^{-m-1}
\]
uniformly for $z\in{[-T_{m},T_{m}]}$. Thus, in particular for $\zeta\in{\mathscr{L}_{\varphi}\cap[-T_{m},T_{m}]}$
condition \eqref{eq:4} implies
\begin{equation} \label{eq:eq2}
\vert \widetilde{Q}_{m}\cdot (f(\zeta)-P_{m}(\zeta))\vert 
\leq \vert \widetilde{Q}_{m}\vert\cdot \frac{1}{2}\tau_{m}^{-m-1}
\leq \frac{1}{2}\vert \widetilde{Q}_{m}\vert^{-m}.
\end{equation}
Combination of \eqref{eq:drauf}, \eqref{eq:eq2} and the triangular inequality yield 
\begin{equation} \label{eq:ende}
\Vert \widetilde{Q}_{m}\cdot f(\zeta)\Vert \leq \vert \widetilde{Q}_{m}\vert^{-m}.
\end{equation}
Now we repeat the procedure with the polynomial $P_{m+1}(z)=c_{0}+\cdots+c_{m+1}z^{m+1}$,
where we have to satisfy the condition \eqref{eq:hbed} for $m$ and $m+1$, which however we
may easily do by choosing any sufficiently small rational $c_{m+2}=1/b_{m+2}$ with $b_{m+1}\vert b_{m+2}$.
Proceeding in this manner, we obtain integer solutions to the estimate \eqref{eq:ende} 
for any $m\geq 1$ and any $\zeta\in{\mathscr{L}_{\varphi}\cap[-T_{m},T_{m}]}$.
Any $\zeta$ belongs to $[-T_{m},T_{m}]$ for all large $m\geq m_{0}(\zeta)$,
hence indeed $\mu(f(\zeta))=\infty$ or $f(\zeta)\in{\mathbb{Q}}$ for any $\zeta\in{\mathscr{L}_{\varphi}}$, 
where $\mu$ denotes the irrationality exponent.
We have to exclude the case $f(\zeta)\in{\mathbb{Q}}$ to infer $f(\zeta)\in{\mathscr{L}}$,
simultaneously for all $\zeta\in{\mathscr{L}_{\varphi}}$.

Assume $f(\zeta)\in{\mathbb{Q}}$ for some $\zeta\in{\mathscr{L}_{\varphi}}$, say $f(\zeta)=l_{1}/l_{2}$
with coprime integers $l_{1},l_{2}$. For $\widetilde{q}_{m}$ as constructed in the proof,
let $\widetilde{p}_{m}/\widetilde{q}_{m}$ be the good approximation to $\zeta$ with
denominator $\widetilde{q}_{m}$, i.e. $\widetilde{p}_{m}$
is the closest integer to $\zeta \widetilde{q}_{m}$. 
Recalling the definition of $\widetilde{q}_{m}$ in \eqref{eq:mitte},
we may assume $(\widetilde{p}_{m},\widetilde{q}_{m})=1$, otherwise we could divide both
$\widetilde{p}_{m},\widetilde{q}_{m}$ by their greatest common divisor and \eqref{eq:mitte} 
still holds (in fact the left hand side is even smaller and the right hand side larger) 
and all above works analogue. Further
say $\widetilde{R}_{m}$ is the closest integer to $\widetilde{Q}_{m}f(\zeta)$ for $m\geq 1$. 
The estimate \eqref{eq:ende} can be written
\begin{equation} \label{eq:lform}
\vert \widetilde{Q}_{m}f(\zeta)-\widetilde{R}_{m}\vert \leq \vert\widetilde{Q}_{m}\vert^{-m}, \qquad  m\geq 1.
\end{equation}
On the other hand, if for some $m$ we have $\widetilde{R}_{m}/\widetilde{Q}_{m}\neq l_{1}/l_{2}$, then 
\begin{equation}  \label{eq:rsl}
\vert \widetilde{Q}_{m}f(\zeta)-\widetilde{R}_{m}\vert=
\left\vert \widetilde{Q}_{m}\frac{l_{1}}{l_{2}}-\widetilde{R}_{m}\right\vert \geq \frac{1}{l_{2}}, \qquad m\geq 1.
\end{equation}
Since both \eqref{eq:lform}, \eqref{eq:rsl} cannot hold for large $m$, we must have
\begin{equation} \label{eq:tomcat}
\frac{\widetilde{R}_{m}}{\widetilde{Q}_{m}}=f(\zeta)=\frac{l_{1}}{l_{2}}, \qquad m\geq m_{0}.
\end{equation}
Since $\widetilde{Q}_{m}=B_{m}\widetilde{q}_{m}^{m}$ and
$\lim_{m\to\infty}\widetilde{q}_{m}=\infty$, it suffices 
to show $\widetilde{R}_{m}$ and $\widetilde{q}_{m}$ are 
coprime for any fixed $m$ to contradict \eqref{eq:tomcat}. 
Due to \eqref{eq:eq2}, $\widetilde{R}_{m}$ equals the closest integer 
to $\widetilde{Q}_{m}P_{m}(\zeta)$ as well.
Hence, recalling \eqref{eq:mitte} and $k_{m}\geq \nu_{0}(P_{m})$, 
Proposition~\ref{relativprim} indeed implies $(\widetilde{R}_{m},\widetilde{q}_{m})=1$.
This contradicts the hypothesis $f(\zeta)\in{\mathbb{Q}}$,
which finishes the proof of $f(\mathscr{L}_{\varphi})\subseteq \mathscr{L}$.

We carry out how the above generalizes to the larger class $\mathscr{L}_{\varphi}^{\ast}$.
We may assume that the sequence $(k_{m})_{m\geq 1}$ tends to infinity, otherwise we can
choose larger values in any step. Thus by definition of $\mathscr{L}_{\varphi}^{\ast}$,
for any $\zeta\in{\mathscr{L}_{\varphi}^{\ast}}$ the estimate \eqref{eq:mitte} has 
a solution $2\leq \widetilde{q}\leq \varphi(k_{m})$ for all large $m\geq m_{0}(\zeta)$. 
Hence we deduce solutions to \eqref{eq:ende} for $m\geq m_{0}(\zeta)$
which guarantees $f(\zeta)\in{\mathscr{L}\cup \mathbb{Q}}$. The exclusion of $f(\zeta)\in{\mathbb{Q}}$
obviously works as for $\zeta\in{\mathscr{L}_{\varphi}}$.

Next we show $f(\mathbb{Q}\setminus\{0\})\subseteq \mathscr{L}$. Let $l_{1}/l_{2}\in{\mathbb{Q}}$
arbitrary and write $B_{m}/b_{j}=d_{m,j}\in{\mathbb{Z}}$ for $m\geq 1$ and $0\leq j\leq m$. Then on the one hand
\[
B_{m}l_{2}^{m}P_{m}(l_{1}/l_{2})=B_{m}l_{2}^{m} \sum_{j=0}^{m} c_{j}\left(\frac{l_{1}}{l_{2}}\right)^{j}
=\sum_{j=0}^{m} d_{m,j}l_{1}^{j}l_{2}^{m-j}=:\mathscr{A}_{m}\in{\mathbb{Z}}
\]
by construction, on the other hand
\[
\vert B_{m}l_{2}^{m}(f(l_{1}/l_{2})-P_{m}(l_{1}/l_{2}))\vert \leq 
\left\vert B_{m}l_{2}^{m} \sum_{j=m+1}^{\infty} c_{j}\left(\frac{l_{1}}{l_{2}}\right)^{j}\right\vert 
\leq (B_{2}l_{2}^{m})^{-m}
\]
for large $m$ by the fast decay of $c_{j}=1/b_{j}=1/B_{j}$. Triangular inequality shows 
that $\mu(f(l_{1}/l_{2}))=\infty$ unless $f(l_{1}/l_{2})\in{\mathbb{Q}}$,
and that $\mathscr{A}_{m}$ is the closest integer to $B_{m}l_{2}^{m}f(l_{1}/l_{2})$.
By virtue of the same principle as in \eqref{eq:rsl}, it suffices to check 
that $\mathscr{A}_{m}/(B_{m}l_{2}^{m})=P_{m}(l_{1}/l_{2})$ is not constant for all $m\geq m_{0}$ to
exclude the case $f(l_{1}/l_{2})\in{\mathbb{Q}}$ and thus $f(l_{1}/l_{2})\in \mathscr{L}$.
However, since $P_{m+1}(z)=P_{m}(z)+c_{m+1}z^{m+1}$, the equality
$P_{m}(l_{1}/l_{2})=P_{m+1}(l_{1}/l_{2})$ for some $m$ implies $c_{m+1}=0$, which is false, 
unless $l_{1}/l_{2}=0$. This yields the assertion.

We check that $f$ has the remaining desired properties.
The expression $1/(m+h)!$ in \eqref{eq:hbed} guarantees that $f$ is an entire function,
which by construction has rational coefficients and is not a polynomial.
Hence it is transcendental as carried out in Section~\ref{definitions}.  
Clearly, this method is flexible enough to provide uncountably many suitable $f$.

It remains to extend the assertion to the derivatives. 
We may assume that in every recursive step the condition $b_{m}\vert b_{m+1}$ is 
strengthened to $m!b_{m}\vert b_{m+1}$. All derivatives of $f$ are then again
of the form $f^{(s)}(z)=\sum_{j\geq 0} (1/b_{j}^{(s)}) z^{j}$ for integers $b_{j}^{(s)}$ with 
the property $b_{j}^{(s)}\vert b_{j+1}^{(s)}$ for all pairs $j\geq 0,s\geq 0$. 
Let $s\geq 0$ be fixed now.
If we define $A_{m}^{(s)}, B_{m}^{(s)}$ for $P_{m}^{(s)}$ the $m$-th Taylor
polynomial of $f^{(s)}$ as in Lemma~\ref{grobabschaetzung},
then by the above $A_{m}^{(s)}=1, B_{m}^{(s)}=b_{m}^{(s)}$ for all $m\geq 0$, as in the case $s=0$.
By construction also $B_{m}^{(t+1)}=(m+1)^{-1}B_{m+1}^{(t)}<B_{m+1}^{(t)}$ for all $m\geq 0, t\geq 0$
and thus $B_{m}^{(s)}<B_{m+s}$. 
Thus if we put $k_{m}^{(s)}:=k_{m+s}$, then similarly to \eqref{eq:sc} the estimate
\begin{equation} \label{eq:sc2}
q^{k_{m}^{(s)}}>m^{2}(T_{m}+1)^{m-1}A_{m}^{(s)}B_{m}^{(s)}q^{m-1}\cdot 2(B_{m}^{(s)}q^{m})^{m}=:q^{m^{2}+m-1}D_{m}^{(s)}
\end{equation}
will be satisfied for all $q\geq 2$ with $D_{m}^{(s)}:=D_{m+s}$. Similarly to \eqref{eq:mitte} we infer
\[
\Vert q\zeta\Vert \leq q^{-k_{m}^{(s)}}
\]
has a solution $q=:\widetilde{q}_{m}^{(s)}$, that may depend on $\zeta$ but
with $2\leq \widetilde{q}_{m}\leq \varphi(k_{m}^{(s)})$ uniformly.
Proceeding further as in the case $s=0$, the analogue of \eqref{eq:eq1} holds again
and with $\widetilde{Q}_{m}^{(s)}:=B_{m}^{(s)}\widetilde{q}_{m}^{(s) m}$
we further obtain 
\begin{equation} \label{eq:drauf2}
\Vert\widetilde{Q}_{m}^{(s)}P_{m}^{(s)}(\zeta)\Vert \leq \frac{1}{2}\widetilde{Q}_{m}^{(s) -m}.
\end{equation}
Moreover, with $\tau_{m}^{(s)}:=\tau_{m+s}$ also
\begin{equation} \label{eq:42}
\vert \widetilde{Q}_{m}^{(s)}\vert\leq \tau_{m}^{(s)}.
\end{equation}
For the estimate of the remainder term, first note that the coefficients $c_{j}^{(s)}$ of $f^{(s)}$ 
satisfy 
\[
c_{j}^{(s)}=\frac{1}{b_{j}^{(s)}}=\frac{1}{b_{j+s}}j(j+1)\cdots (j+s-1)
\leq (j+s)^{s}\frac{1}{b_{m+s}}=(j+s)^{s}c_{m+s}.
\]
Hence
\[
\vert f(z)^{(s)}-P_{m}^{(s)}(z)\vert =\left\vert \sum_{h=1}^{\infty} c_{m+h}^{(s)}z^{m+h}\right\vert
\leq \sum_{h=1}^{\infty}\vert c_{m+h}^{(s)}\vert T_{m}^{m+h}
\leq \sum_{h=1}^{\infty} (m+h+s)^{s}\vert c_{m+h+s}\vert T_{m}^{m+h}
\]
uniformly for $z\in{[-T_{m},T_{m}]}$. If we strengthen the condition \eqref{eq:hbed}
in any inductive step by replacing $\tau^{-m}$ by $\tau^{-m^{2}}$ if necessary,
from the fast decay of $(c_{m})_{m\geq 1}$
and since $s$ is fixed, it clearly follows that at least for large $m$ the above can be bounded by
\[
\vert f(z)^{(s)}-P_{m}^{(s)}(z)\vert \leq \sum_{h=1}^{\infty} (m+h+s)^{s}\vert c_{m+h+s}\vert T_{m}^{m+h}
\leq \frac{1}{2}\tau_{m}^{(s) -m-1}.
\]
In combination with \eqref{eq:42} for large $m$ again
\[
\vert \widetilde{Q}_{m}^{(s)}\cdot (f^{(s)}(\zeta)-P_{m}^{(s)}(\zeta))\vert 
\leq \vert \widetilde{Q}_{m}^{(s)}\vert\cdot \frac{1}{2}\tau_{m}^{(s) -m-1}
\leq \frac{1}{2}\vert \widetilde{Q}_{m}^{(s)}\vert^{-m},
\]
and together with \eqref{eq:drauf2} and triangular inequality eventually 
\begin{equation} \label{eq:ende2}
\Vert \widetilde{Q}_{m}^{(s)}\cdot f^{(s)}(\zeta)\Vert \leq \vert \widetilde{Q}_{m}^{(s)}\vert^{-m}.
\end{equation}
As this holds for all $\zeta\in{\mathscr{L}_{\varphi}}$ and
large $m$ indeed $f^{(s)}(\mathscr{L}_{\varphi})\subseteq \mathscr{L}$.
The generalization to $\mathscr{L}_{\varphi}^{\ast}$ such as the
proof of $f^{(s)}(\zeta)\notin{\mathbb{Q}}$ and $f^{(s)}(\mathbb{Q}\setminus \{0\})\subseteq \mathscr{L}$
works very similar to the case $s=0$.
\end{proof}

We give several remarks.

\begin{remark} \label{qbild}
The assertion $f(\mathbb{Q}\setminus\{0\})\subseteq \mathscr{L}$ implies
$f(\mathbb{Q}\setminus\{0\})$ is a purely transcendental set, see Section~\ref{beginn}.
Observe the contrast to Theorem~\ref{maillet}, Theorem~\ref{anfang}, Corollary~\ref{bfang} and 
Theorem~\ref{moreira} where we had $f(\mathbb{Q})\subseteq \mathbb{Q}$. Moreover,
since an function $f$ algebraic over $\overline{\mathbb{Q}}$ satisfies $S_{f}=\overline{\mathbb{Q}}$, 
this leads to a proof that all constructed functions are transcendental over the base field 
$\overline{\mathbb{Q}}$ instead of $\mathbb{C}$. This is weaker 
but avoids the rather deep Great Picard~Theorem, see Section~\ref{definitions}. 
\end{remark}

\begin{remark}
We only needed $\zeta\in{\mathscr{L}_{\varphi}}$ to obtain a
uniform  bound of $\widetilde{q}_{m}$ in \eqref{eq:mitte}. 
If we weaken this to $\zeta\in{\mathscr{L}}$, we further have
no uniform bound in \eqref{eq:4} which is needed to bound
the left hand side in \eqref{eq:eq2}, even restricting to $\zeta$ 
in a given compact interval. 
\end{remark}

\begin{remark}
For any finite set $\{\zeta_{1},\zeta_{2},\ldots,\zeta_{u}\}\subseteq \mathscr{L}^{u}$,
the proof of Theorem~\ref{haupt} provides a method of constructing entire transcendental 
functions $f$ that map all $\zeta_{j}$ simultaneously to elements of $\mathscr{L}$. 
It suffices to define the involved function $\varphi$ as the pointwise maximum
of the individual minimum functions for $\zeta_{j}$, as carried out
subsequent to Definition~\ref{psi}. However, such functions $f$ 
can alternatively be constructed with the Weierstrass factorization Theorem,
see Chapter~7 paragraph~5 in~\cite{conway}.        
\end{remark}
 
It is evident that Theorem~\ref{haupt} becomes more interesting
the faster the function $\varphi$ tends to infinity. See Section~\ref{vieledefs}
for examples of $\varphi$ inducing large sets $\mathscr{L}_{\varphi}$.
From Proposition~\ref{hilfspro} and Theorem~\ref{haupt} we further infer a last corollary.

\begin{corollary}
Let $\Lambda$ be any function as in Definition~{\upshape\ref{defilio}}. Then
there exist uncountably many entire transcendental functions $f$ with
$f(\mathscr{L}^{s,\Lambda})\subseteq f(\mathscr{L}^{ss,\Lambda})\subseteq \mathscr{L}$.
\end{corollary}

\begin{proof}
Given $\Lambda$, by Proposition~\ref{hilfspro} we can choose $\varphi$ such that
$\mathscr{L}^{s,\Lambda}\subseteq \mathscr{L}^{ss,\Lambda}\subseteq \mathscr{L}_{\varphi}^{\ast}$.
By virtue of Theorem~\ref{haupt} on the other hand we can find suitable $f$ such that
$f(\mathscr{L}_{\varphi}^{\ast})\subseteq \mathscr{L}$. Thus
$f(\mathscr{L}^{s,\Lambda})\subseteq f(\mathscr{L}^{ss,\Lambda})\subseteq
 f(\mathscr{L}_{\varphi}^{\ast})\subseteq \mathscr{L}$.
\end{proof}

\section{The converse problem: $f(\mathscr{L})\cap \mathscr{L}=\emptyset$}  \label{sek4}

To this point, we have tried to find examples of analytic functions with a large 
set $f(\mathscr{L})\cap \mathscr{L}$. This suggests the following converse problem.

\begin{problem} \label{problem}
Are there non-constant analytic functions $f$ with real coefficients 
such that $\mathscr{L}\cap f(\mathscr{L})= \emptyset$?
\end{problem}

A negative answer can be readily inferred from a recent result on Liouville numbers~\cite{kumar}, 
which bases solely on the topological property of $\mathscr{L}$ being a $G_{\delta}$ dense set. 

\begin{theorem}[Kumar, Thangadurai, Waldschmidt]  \label{than}
Let $I$ be a non-empty open interval of $\mathbb{R}$ and let 
$(f_{n})_{n\geq 0}$ be a sequence of real continuous functions on $I$ 
which are nowhere locally constant. Then there exists an uncountable $G_{\delta}$-set 
$E\subseteq \mathscr{L}\cap I$ such that $f_{n}(E)\subseteq \mathscr{L}$, for all $n\geq 0$.
\end{theorem}

See also~\cite{alni},~\cite{erdos},~\cite{rieger},~\cite{schwarz} and
~\cite{burger} (however, as pointed out in the \textit{MathSciNet} review, 
the proof in~\cite{burger} has a small gap and it does not work in general. 
See Silva~\cite{silva} for a recent slightly weaker result). 
As a corollary we indeed obtain

\begin{theorem}        \label{d1}  
Let $I\subseteq \mathbb{R}$ be a non-empty open interval and
$f:I\mapsto \mathbb{R}$ be a non-constant analytic function.
Then there exists an uncountable set
$E\subseteq \mathscr{L}\cap I$, such that $f^{(s)}(E)\subseteq \mathscr{L}$ for all $s\geq 0$.
\end{theorem}

\begin{proof}
Apply Theorem~\ref{than} with $f_{n}:=f^{(n)}$ for $n\geq 0$, and note that the condition
clearly holds due to Identity~Theorem. It claims that an analytic function on a connected open set $I$ is determined
by its values on a set with limit point in $I$. In particular the constant $0$ function is the only
analytic function that takes the value $0$ on a set with limit point. Hence any entire function which is constant
on some real interval must already be constant on $\mathbb{C}$, or equivalently a non-constant entire function
is not constant on any real interval.
\end{proof}

\section{The set $f(\mathscr{L})\cap \mathscr{L}$ for the functions $f(z)=z^{a/b}$}  \label{hochab}

Theorem~\ref{maillet} implies $f(z)=z^{k}$ for an integer 
$k\neq 0$ satisfies $f(\mathscr{L})\subseteq \mathscr{L}$.
A more general class of functions one may consider is $f(z)=z^{a/b}$ for rational numbers $a/b$.
In certain subsets of $\mathbb{C}$ there might be several representatives of $f$, however 
we are only interested in the real representative $f:(0,\infty)\mapsto (0,\infty)$.
Any such function $f$ is algebraic even over the base 
field $\mathbb{Q}$ as $f(z)^{b}-z^{a}=0$, in particular $S_{f}=\overline{\mathbb{Q}}$.
Further at any $s\in{(0,\infty)}$ the function $f$ admits a local power series 
expansion $f(z)=c_{0}+c_{1}(z-s)+c_{2}(z-s)^{2}+\cdots$ with radius of
convergence $s$. 
Moreover, one checks that the power series expansion 
at a point $s\in{\overline{\mathbb{Q}}\cap \mathbb{R}}$ has coefficients 
$c_{j}\in{\overline{\mathbb{Q}}\cap \mathbb{R}}$.

The first result is an easy observation and more for sake of completeness.
It provides explicit constructions of $\zeta$ fixed under given $f$, as in Theorem~\ref{d1}.

\begin{theorem} \label{easycheesy}
For integer parameters $a\neq 0,b\neq 0$ let $f_{a,b}(z)=z^{a/b}$. 
Further let $I\subseteq (0,\infty)$ with non-empty interior.
Then there exist uncountably many $\zeta\in{\mathscr{L}}$ such that $f_{a,b}(\zeta)\in{\mathscr{L}}$
simultaneously for all $a,b$.
Moreover, for fixed $a,b$, uncountably many such $\zeta\in{\mathscr{L}\cap I}$ 
can be explicitly constructed.
\end{theorem}

\begin{proof}
The first assertion follows from Theorem~\ref{than} with $(f_{n})_{n\geq 1}$ 
any enumeration of the set of functions $f_{a,b}$. 
Now consider $a,b$ fixed and let $f:=f_{a,b}$. Due to Theorem~\ref{maillet} we may assume $a>0, b>0$.
Clearly, if we take arbitrary $\zeta^{\prime}\in{\mathscr{L}\cap(0,\infty)}$ and put $\zeta=\zeta^{\prime b}$,
then Theorem~\ref{maillet} implies $\zeta\in{\mathscr{L}}$ and
$f(\zeta)=\zeta^{a/b}=\zeta^{\prime a}$ is in $\mathscr{L}$. Moreover, since $x\mapsto x^{b}$
induces a homeomorphism on $(0,\infty)$, the suitable set 
$\mathscr{L}^{b}:=\{\zeta^{b}:\zeta\in{\mathscr{L}}\}$ inherits the property
of being uncountable in any positive interval from the analogue property of $\mathscr{L}$.
\end{proof}

Now we state the main result of Section~\ref{hochab}, which was already indicated in Section~\ref{outline}.

\begin{theorem}  \label{rati}
Let $f_{a,b}$ and $I$ as in Theorem~{\upshape\ref{easycheesy}}.
Then there exist uncountably many $\zeta\in{\mathscr{L}\cap I}$ such that $f_{a,b}(\zeta)\in{\mathscr{L}}$
if and only if $a/b$ is an integer. 
Suitable $\zeta$ can be explicitly constructed.
In particular, for any fixed coprime $a,b$ with $\vert b\vert\geq 2$, we have 
$f_{a,b}(\mathscr{L})\cap \mathscr{L}$ is uncountable but $f_{a,b}(\mathscr{L})\nsubseteq \mathscr{L}$.
\end{theorem}

We compare Theorem~\ref{rati} with a result connected to $U$-numbers 
in Mahler's classification introduced in Section~\ref{beginn}. 
Theorem~7.4 and its proof in~\cite{ybu} provides an explicit example of a number
$\zeta_{0}$ whose $m$-th root is a $U_{m}$-number for any integer $m\geq 1$. 
This implies Theorem~\ref{rati} for $a=1$, and is in fact stronger for $b>2$ since the latter
only yields that $\zeta^{a/b}$ is a $U_{l}$-number for some $2\leq l\leq b$. 
In contrast to Theorem~7.4 in~\cite{ybu}, Theorem~\ref{rati} provides no information on 
approximation by algebraic irrational numbers.
However, it seems that the general assertion of Theorem~\ref{rati} cannot be deduced entirely from
Theorem~7.4 in~\cite{ybu} or related results. 

Theorems~\ref{easycheesy} and~\ref{rati}, in view of Theorem~\ref{maillet}, suggest the following conjecture.

\begin{conjecture}
Let $I$ and $f_{a,b}$ be as in Theorem~\ref{easycheesy}. Further let $A\subseteq \mathbb{Q}\setminus\{0\}$ be
arbitrary with the properties $1\in{A}$ and for any element 
of $A$ any non-zero integral multiple belongs to $A$ as well.
Then there exist uncountably many $\zeta\in{\mathscr{L}\cap I}$ with the 
property that $f_{a,b}(\zeta)\in{\mathscr{L}}$ if $a/b\in{A}$ and
$f_{a,b}(\zeta)\notin{\mathscr{L}}$ if $a/b\notin{A}$.  
\end{conjecture}

Obviously Theorems~\ref{easycheesy} and~\ref{rati} provide the extremal cases
$A=\mathbb{Q}\setminus\{0\}$ and $A=\mathbb{Z}\setminus\{0\}$. 
If we drop the condition $1\in{A}$, then the conjecture might be true for some $\zeta\in{I}$
not necessarily in $\mathscr{L}$.
We collect some ingredients for the proof of Theorem~\ref{rati} in the next section. 

\subsection{Preparatory results}
It was known by Maillet~\cite{maillet} that the $b$-th root of $\zeta\in{\mathscr{L}}$ 
is a Liouville number if and only if among the convergents of $\zeta$ there are infinitely
many $b$-th powers of rationals. We carry out his main argument for the necessity of this condition in
the following more general Lemma~\ref{inf}, which in particular will allow us to establish
effective bounds in Corollary~\ref{dadkorollar}.

\begin{lemma} \label{inf}
Let $a/b$ be a rational number in lowest terms. 
Suppose $\zeta\in{\mathscr{L}}$ and $\zeta^{a/b}\in{\mathscr{L}}$. Then
for any $\eta>0$ the inequality 
\begin{equation} \label{eq:exponent}
\vert q^{b}\zeta^{a}-p^{b}\vert \leq q^{-\eta}
\end{equation}
has a solution in coprime integers $p,q$. Moreover, if $\eta>b$ is fixed
and $q$ is large, then $p^{b}/q^{b}$ is a convergent of $\zeta^{a}$.
\end{lemma}

\begin{proof}
Assume for a real number $\alpha$ and a positive integer $k$ the estimate \eqref{eq:lio2} is satisfied.
This implies \eqref{eq:lio} 
with a constant $D(k,\alpha)$ depending only on $k$ and $\alpha$. Further observe that if we have
\begin{equation} \label{eq:bedi}
q^{-\nu+k-1}<\frac{1}{2D(k,\alpha)q},
\end{equation} 
then Theorem~\ref{prop} and \eqref{eq:lio} imply
for large $q$ that $p^{k}/q^{k}$ is a convergent of $\alpha^{k}$.
Obviously for fixed $k,\alpha$
the estimate \eqref{eq:bedi} is satisfied for any $\nu>k$ and all large $q\geq q_{0}(\nu)$.

Suppose $\zeta$ and $\zeta^{a/b}$ both belong to $\mathscr{L}$ for some suitable $a,b$. 
The above argument with $k=b$, $\alpha=\zeta^{a/b}$ shows that for 
arbitrarily large $\eta$ the estimate \eqref{eq:exponent}
has a solution $(p,q)\in{\mathbb{N}^{2}}$ with $p^{b}/q^{b}$
a convergent of $\zeta^{a}$.  
\end{proof}

In the proof of the more technical case $a>1$ of Theorem~\ref{rati},
we will need the following basic result Lemma~\ref{minkkoro}. 
It can be derived by the combination of Theorem~\ref{prop}
and Proposition~4.6 in~\cite{schl} (or if one prefers directly from Minkowksi's second
lattice point Theorem, see Section~1 in~\cite{sumschm}).

\begin{lemma} \label{minkkoro}
Let $\alpha\in{\mathbb{R}}$.
For any parameter $Q>1$, there cannot be two linearly independent integral
solution pairs $(x,y)$ to the system 
\[
\vert x\vert \leq Q,  \qquad \vert\alpha x-y\vert <\frac{1}{2Q}.
\]
Moreover, if $(x,y)$ is a solution for some $Q$, then $y/x$ must be a convergent of $\alpha$.
\end{lemma}

It will be convenient to apply Dirichlet's~Theorem on primes in arithmetic progressions~\cite{dirichlet}
to shorten the proof of Theorem~\ref{rati}, although more elementary methods would work out as well.
See also Remark~\ref{darnach}. 

\begin{theorem}[Dirichlet]  \label{dps}
Let $A,B$ be coprime positive integers.
Then the arithmetic progression $a_{n}=An+B$ contains infinitely many prime numbers.
\end{theorem}

Now we are ready for the proof of Theorem~\ref{rati}.

\subsection{Proof of Theorem~\ref{rati}}

As in Theorem~\ref{easycheesy} we may assume $a>0, b>0$.
If $a/b$ is an integer and $\zeta\in{\mathscr{L}}$, 
then $f_{a,b}(\zeta)\in \mathscr{L}$ by Theorem~\ref{maillet}.
Thus it suffices to construct $\zeta\in{\mathscr{L}}$ with $f_{a,b}(\zeta)\notin{\mathscr{L}}$ 
simultaneously for all coprime pairs $a,b$ with
$b\geq 2$. At first we drop the restriction $\zeta\in{I}$.
Due to Lemma~\ref{inf}, it suffices to find $\zeta\in{\mathscr{L}}$ 
such that for each pair $a,b$ we can find $\eta=\eta(a,b)>b$
such that $\zeta^{a}$ has no convergent of the form $p^{b}/q^{b}$ 
for which \eqref{eq:exponent} has a solution for $\eta=\eta(a,b)$,
to infer $\zeta^{a/b}\notin{\mathscr{L}}$. 

We construct such $\zeta$. We want that the partial quotients of $\zeta$ are
rapidly increasing and all denominators of convergents of $\zeta$ are prime numbers. 
With the notation as above, 
suppose the partial denominators $r_{0},r_{1},\ldots,r_{g}$ are constructed with the property 
that the denominators of all convergents $s_{1}/t_{1},\ldots, s_{g}/t_{g}$ are 
primes. Subsequent to \eqref{eq:contf} we remarked that $t_{g-1},t_{g}$ are coprime. 
By Theorem~\ref{dps} and \eqref{eq:contf},
we may choose arbitrarily large $r_{g+1}$ such that $t_{g+1}$ is prime.
We require $r_{g+1}\geq t_{g}^{g}$, and for technical reasons 
the sequence $r_{n}$ should moreover grow fast enough that if $\nu_{n}$ is
defined by by $\vert\zeta t_{n}- s_{n}\vert= t_{n}^{-\nu_{n}}$, then $\nu_{n+1}>\nu_{n}$ in any step.
By Theorem~\ref{cfe} obviously $\lim_{n\to\infty} \nu_{n}=\infty$, such that
this procedure indeed leads to $\zeta\in{\mathscr{L}}$. 
We have to show that $\zeta$ has the requested property. 
Throughout the remainder of the proof let $\delta>0$ be
arbitrarily small but fixed.

First let $a=1$. In this case it suffices to put $\eta(1,b)=b+\delta$ and
observe that by construction all convergents of $\zeta^{a}=\zeta$
have prime denominators and hence no convergent is of the 
form $p^{b}/q^{b}$ for $b\geq 2$.

Now let $a\geq 2$. We show that the inequality
\begin{equation} \label{eq:exp2}
\vert x\zeta^{a}-y\vert \leq x^{-a-\delta}
\end{equation}
can hold for $(x,y)\in{\mathbb{N}^{2}}$ with large $x$ only in case of $(x,y)$
an integral multiple of some $(q^{\prime a},p^{\prime a})$, where $p^{\prime a}/q^{\prime a}$ is a convergent 
of $\zeta^{a}$ in lowest terms. More precisely, $(p^{\prime},q^{\prime})=(s_{n},t_{n})$ for 
some $n$, with $s_{n},t_{n}$ as above. Assume this is true.
Let $\eta=\eta(a,b)=\max\{a+\delta,b+\delta\}$. Assume for this choice of $\eta$ there exist
solutions of \eqref{eq:exponent}, that must be
convergents of $\zeta^{a}$ of the form $p^{b}/q^{b}$ by Lemma~\ref{inf}. On the other hand, by the above
observation and the choice of $\eta$, these solutions must at the same time have a representation
as a quotient of $a$-th powers of integers $p^{\prime a}/q^{\prime a}$.
Since $a,b$ are coprime and $q^{\prime}=t_{n}$ is a prime number, this 
is clearly impossible, contradiction.
This yields again an indirect proof of $\zeta^{a/b}\notin{\mathscr{L}}$. 

It remains to check the assertion above.
We have to check that for $(x,y)\in{\mathbb{N}^{2}}$ with large $x$ and 
linearly independent to any $(s_{n}^{a},t_{n}^{a})$, we cannot have \eqref{eq:exp2}.
Consider large $x$ fixed and let $N$ be the index such that $t_{N}\leq x<t_{N+1}$. 
Recall all $s_{n}/t_{n}$ are very good approximations to $\zeta$.
By construction of $\zeta$ and definition of $\nu_{n}$,
in particular we have $\vert \zeta t_{n+1}-s_{n+1}\vert< t_{n+1}^{-\nu_{n}}$.
Then similar to \eqref{eq:lio} we can write
\begin{align} 
\vert t_{N}^{a}\zeta^{a}-s_{N}^{a}\vert=
\vert t_{N}\zeta-s_{N}\vert\cdot\vert t_{N}^{a-1}\zeta^{a-1}+\cdots+s_{N}^{a-1}\vert
&\leq D(a,\zeta)t_{N}^{-\nu_{N}+a-1}   \label{eq:liovi1}  \\
\vert t_{N+1}^{a}\zeta^{a}-s_{N+1}^{a}\vert=
\vert t_{N+1}\zeta-s_{N+1}\vert\cdot\vert t_{N+1}^{a-1}\zeta^{a-1}+\cdots+s_{N+1}^{a-1}\vert
&\leq D(a,\zeta)t_{N+1}^{-\nu_{N}+a-1}.     \label{eq:liovi2}
\end{align}
Moreover $t_{N+1}\asymp t_{N}^{\nu_{N}}$ in view of \eqref{eq:contf} and Theorem~\ref{cfe}. 
We distinguish two cases.

Case 1: $t_{N}\leq x< t_{N}^{a}$. We apply Lemma~\ref{minkkoro}, 
with $Q:=t_{N}^{a}$. Since $(s_{N}^{a},t_{N}^{a})$
leads to a good approximation for $\zeta^{a}$ by \eqref{eq:liovi1}, 
there cannot be another vector $(u,v)\in{\mathbb{N}^{2}}$
linearly independent to $(s_{N}^{a},t_{N}^{a})$ with $u<t_{N}^{a}$ that leads to a good approximation.
As the condition $x<t_{N}^{a}$ is satisfied by assumption, 
Lemma~\ref{minkkoro} more precisely yields that $\vert \zeta^{a} x-y\vert>(1/2)t_{N}^{-a}$.
Since $t_{N}\leq x$, for large $x$ (or $N$) we conclude 
\[
\vert \zeta^{a} x-y\vert>(1/2)t_{N}^{-a}\geq (1/2)x^{-a}>x^{-a-\delta},
\]
indeed a contradiction to \eqref{eq:exp2}.

Case 2: $t_{N}^{a}\leq x<t_{N+1}$. First assume $x$ is close to $t_{N+1}$, more precisely
$t_{N+1}^{1-\epsilon}\leq x<t_{N+1}$ for $\epsilon\in{(0,\delta/(a+\delta))}$. Then 
we may use the same argument as in case 1 with $Q=t_{N+1}^{a}$ instead of $Q=t_{N}^{a}$,
since $\vert \zeta^{a} x-y\vert>(1/2)t_{N+1}^{-a}>x^{-a-\delta}$ is still valid. 
So we may assume $t_{N}^{a}\leq x<t_{N+1}^{1-\epsilon}$.
In this case we apply Lemma~\ref{minkkoro} with $Q:=x$. Assume \eqref{eq:exp2}
holds. Then $(x,y)$ is a pair with $\vert \zeta^{a} x-y\vert<(1/2)Q^{-1}$, so by
Lemma~\ref{minkkoro} there cannot be another such pair linearly independent to $(x,y)$.
However, we show $(s_{N}^{a},t_{N}^{a})$ satisfies the inequality as well.
Recall $t_{N+1}\asymp t_{N}^{\nu_{N}}$ such that $Q=x<t_{N+1}^{1-\epsilon}$ yields
$Q^{1/[(1-\epsilon)\nu_{N}]}\ll t_{N}$. By \eqref{eq:liovi1}  
we infer
\[
\vert t_{N}^{a}\zeta^{a}-s_{N}^{a}\vert \ll t_{N}^{-\nu_{N}+a-1}
\ll Q^{\frac{-\nu_{N}+a-1}{(1-\epsilon)\nu_{N}}}\ll Q^{-\frac{1}{1-\epsilon}}
\]
for large $N$ as $\nu_{N}$ is then large too.
Since $1/(1-\epsilon)>1$ the right hand side is indeed smaller than $(1/2)Q^{-1}$
for large $x=Q$ and the contradiction again shows \eqref{eq:exp2} is false. 

Finally, we may allow the continued fraction expansion of $\zeta$ to start
with arbitrary $[r_{0};r_{1},r_{2},\ldots,r_{l}]$ and then start the above procedure.
Hence the method is flexible enough to guarantee uncountably many 
suitable $\zeta$ in any subinterval of $(0,\infty)$. This completes the proof.

\begin{remark} \label{darnach}
The constructed $\zeta\in{\mathscr{L}}$ in the proof of Theorem~\ref{rati}
are strong Liouville numbers, see Definition~\ref{defilio}.
Indeed, the method of the proof for $a\geq 2$ with Lemma~\ref{minkkoro} requires that there are no large gaps
between denominators of convergents with very good approximation to $\zeta$. 
Conversely, the proof basically does work for any semi-strong Liouville number
for which no convergent is of the form $p^{b}/q^{b}$ for some $b\geq 2$.     
Recall that for $a=1$, it was already known by Maillet that a sufficient condition is
that no convergent $p/q$ of $\zeta$ 
is of the form $p^{b}/q^{b}$ for $b\geq 2$, which is rather easy to construct. 
\end{remark}

The proof of Theorem~\ref{rati} provides explicit upper bounds for the irrationality 
exponent of $\zeta^{a/b}$ for the involved $\zeta\in{\mathscr{L}}$.

\begin{corollary} \label{dadkorollar}
Let $f_{a,b}(z)$ as in Theorem~~{\upshape\ref{easycheesy}} and
$\zeta\in{\mathscr{L}}$ be constructed as in the proof of Theorem~{\upshape\ref{rati}}. Then
$f_{a,b}(\zeta)\in{\mathscr{L}}$ for $a/b$ an integer but
$\mu(f_{a,b}(\zeta))\leq \max\{\vert a\vert,\vert b\vert\}+\vert b\vert$
simultaneously for all $a,b$ for which $a/b$ is not an integer.
\end{corollary}

\begin{proof}
If $a/b$ is an integer then the assertion follows from Theorem~\ref{maillet} 
as already observed in Theorem~\ref{rati}.
Thus, and since $\mu(\alpha^{-1})=\mu(\alpha)$, we can restrict to $a>0, b>0$ and $a/b$ not an integer.
Let $a=1$. Indeed, the fact that \eqref{eq:exponent} has no (large) solution for 
$\eta=b+\delta$, implies that \eqref{eq:lio2} has no (large) 
solution for $\nu=(b+\delta)+(b-1)=2b-1+\delta$. With $\delta \to 0$ 
and adding $1$ taking into account the transition from linear forms to fractions, we obtain the bound.
The same argument can be applied for $a\geq 2$ with $\eta(a,b)=\max\{a+\delta,b+\delta\}$.
\end{proof}


\section*{Acknowledgements}
The first author is supported by the grant CNPq and FAP-DF and the second author 
is supported by the Austrian Science Fund FWF grant P24828.

\end{document}